\documentclass{amsart}
\usepackage{latexsym,amsmath,amssymb,amscd}

\newtheorem{theorem}{Theorem}

\theoremstyle{definition}

\theoremstyle{remark}

\numberwithin{equation}{section}

\def\mbe{M_b(E)}

\def\cbe{C_b(E)}

\def\mbu{M_b(U)}

\def\bdu{\partial{U}}
\def\ck{C(K)}
\def\cm{C(M)}

\def\cbdu{C(\partial{U})}
\def\fil{\Phi}

\def\px{\mu_x}

\def\fixpsi{\mathcal{T}(\Psi)}

\def\lmf{C(\Psi)}

\def\hk{h_k}
\def\h2{h_2}
\def\lm{\lambda}
\def\gm{\gamma}

\def\lmfz{C(\Psi)_0}

\begin{document}

\title{Asymptotic stability\\
for a class of Markov semigroups}
\author{Bebe Prunaru}

\address{Institute of Mathematics ``Simion Stoilow''
  of the Romanian
Academy,
P.O. Box 1-764,
 RO-014700 Bucharest, 
Romania}
\email{Bebe.Prunaru@imar.ro}


\begin{abstract}

Let $U\subset K$ be an open and dense subset of a compact metric space and let 
$\{\Phi_t\}_{t\ge0}$ be a Markov semigroup on the space of bounded Borel measurable functions on $U$ with the strong Feller property. Suppose that for each $x\in\bdu$ there exists a barrier $h\in C(K)$ at $x$ such that $\Phi_t(h)\ge h$ for all $t\ge0$. Suppose also that every   real-valued $g\in C(K)$ with $\Phi_t(g)\ge g$ for all $t\ge0$ and which attains its global maximum at a point inside $U$ is constant. Then for each  $f\in C(K)$ there exists the uniform limit  $F=\lim_{t\to\infty}\Phi_t(f)$. Moreover $F$ is continuous on $K$, agrees with $f$ on $\partial{U}$ and $\Phi_t(F)=F$ for all $t\ge0$.  

\end{abstract}

\maketitle

Let $E$ be a locally compact Hausdorff  space and let  $\mbe$ be the space of all complex-valued, bounded and Borel measurable  functions on $E$. Let also $\cbe$ be the set of all continuous functions in $\mbe$.

A linear map
$ \fil:\mbe\to\mbe$ 
is called a Markov operator if
for each $x\in E$ there exists a  probability Borel measure
 $\px$ on $E$ such that 
 $$\fil(f)(x)=\int f d\px \quad     \forall f\in\mbe.$$
 
 A Markov operator $\Phi$ is said to have the strong Feller property if 
 $\fil(f)\in\cbe$  for every $f\in\mbe$. 
 The main result of this paper is the following:

 \begin{theorem}\label{main}
 
 Let $K$ be a compact metric space and let $U\subset K$ be a dense open subset
 such that $\bdu$ contains at least two points.
 Let $\ck$ be the space of all continuous complex-valued functions on $K$ and let $\cbdu$ be the corresponding space for $\bdu$.
 
 Let $\fil:\mbu\to\mbu$
  be  a Markov operator with the strong Feller property.
  Suppose that $\fil$ satisfies the following  conditions:
 
 \begin{itemize}

  \item[(A)]  For each point $x\in\bdu$ there exists $h\in\ck$
   such that $h(x)=0$, $h(y)<0$ for all $y\in K\backslash\{x\}$ 
   and $\fil(h_U)\ge h_U$  on $U$, where $h_U$ is the restriction 
  of $h$ to $U$;

 \item[(B)]   If   $g\in\ck$ is  a real valued function  with
$\fil(g_U)\ge g_U$ and if there exists  $z\in U$ such that 
$g(z)=\max\{g(x): x\in K\}$ 
 then $g$ is constant on $K$.

 \end{itemize}

 Then, for each $f\in\cbdu$ there exists a unique 
    function $G\in\ck$ such that 
 $\fil(G_U)=G_U$ 
 and  $G(x)=f(x)$ for all $x\in\bdu$. 
 Moreover,  if  $F\in\ck$  is an arbitrary continuous 
 extension of $f$ to $K$ then the sequence
 $\{\fil^n(F_U)\}_n$ converges uniformly on $U$ to $G_U$.

\end{theorem}

\begin{proof}

 (1) First of all, it can be proved, exactly as in Proposition 1.3 from \cite{AE},
  that the boundary condition (A) implies the following.
For each $f\in\ck$ and for each $x\in\bdu$
   $$\lim\sb{y\to x}\sup\sb{n\ge1}|(\fil^n(f_U))(y)-f(x)|=0.$$

This enables us  to define a map 
$$\Psi:\ck\to\ck$$
 by setting, for each $f\in\ck$, 
$\Psi(f)(x)=\fil(f_U)(x)$ for  $x\in U$  
and $\Psi(f)(x)=f(x)$ for  $x\in\bdu.$

 (2) We  show  that  if $f\in\ck$ is a real-valued function such that 
 $\fil(f_U)\ge f_U$ on $U$, then 
the sequence  $\{\Psi^n(f)\}$ converges uniformly  on $K.$
Indeed, this sequence is monotone increasing and if $g$ is its pointwise limit then 
$g$ is Borel measurable and 
$\fil(g_U)=g_U$.
Since $\fil$ has the strong Feller property, it follows that $g_U$ is continuous on $U$.
Moreover (1) implies that 
$$\lim\sb{y\to x}g(y)=g(x)$$ for every $x\in\bdu.$
This shows that  $g\in\ck$
and  Dini's theorem shows that the convergence is uniform on $K$.

 (3) Let 
 $$\fixpsi=\{h\in\ck : \Psi(h)=h\}$$
 and let $A$ be the norm closed subalgebra of $\ck$ generated by $\fixpsi$.
 Let $\lmf$ be the set of all $f\in\ck$ for which the sequence
 $\{\Psi^n(f)\}$ is uniformly convergent on $K$ and denote $\pi(f)$ its limit.
 Let also 
 $$\lmfz=\{f\in\lmf : \pi(f)=0\}.$$
 We will show that  $A\subset\lmf$.

 Let  $h\in\fixpsi$.
 Then 
 $\Psi(|h|^2)\ge |h|^2$
 therefore (2) shows that $|h|^2\in\lmf$
 and also that $\pi(|h|^2)-|h|^2\ge0$.
 This also shows that the norm closed ideal of $\ck$ generated by all the functions of the form $\pi(|h|^2)-|h|^2$ with $h\in\fixpsi$ is contained in $\lmfz$.  
 Indeed it  is easy to see that if $f\in\lmfz$ is non-negative on $K$ 
  then $fg\in\lmfz$   
 for every $g\in\ck$.

 We shall now prove that for any finite set of $k$ functions 
 from $\fixpsi$ their product 
belongs to $\lmf$. 
Let  $k=2$. 
 If $h_1,\h2\in\fixpsi$ then 
 $$h_1\h2=(1/4)\sum\sb{m=0}\sp{3}i^m |g_m|^2$$
 where $i=\sqrt{-1}$
 and
 $g_m=(h_1+i^m\bar{\h2})$.
 Since $g_m\in\fixpsi$ we see that 
 $h_1\h2\in\lmf$.

 Let $k\ge3$ and assume that  every product of at most $k-1$ elements 
 from $\fixpsi$ belongs to $\lmf$.
   Let $h_1,\dots,\hk$ in $\fixpsi$ and let $g=h_1\cdot\dots\cdot\hk$.
  Then
  $$g=
  (h_1\h2-\pi(h_1\h2))\cdot h_3\cdot\dots\cdot\hk
  +\pi(h_1\h2)\cdot h_3\cdot\dots\cdot\hk.$$
  By what we have already proved the first summand belongs to $\lmfz$ and the second belongs, by the induction hypothesis, to $\lmf$.         
 This shows that $A\subset\lmf$.

  (4) Consider the map 
 $$\rho:A\to\cbdu$$
 that takes any  $f\in A$ into its restriction to $\bdu$. 
It turns out that $\rho$  is onto.
To see this, we first observe that  the boundary condition (A) 
 together with (2) implies  that for each $x\in\bdu$
  there exists $g\in\ck$ such that  
  $g(x)=0$, $g(y)<0$ for every $y\in\bdu\backslash\{x\}$ and $\fil(g_U)=g_U$.
 In particular $g\in\fixpsi$.
 This shows that the range of $\rho$ separates the points of $\bdu$ therefore $\rho$ is onto.
This proves the existence  part of the theorem.
Uniqueness follows easily from (B).

 (5) We shall denote 
$$\theta:\ck\to\ck$$
the map which takes a function $f\in\ck$ into the uniquely determined function in $\fixpsi$ which agrees with $f$ on $\bdu$.
It follows that for each $f\in\lmf$ we have $\theta(f)=\pi(f)$.
In particular, $\theta(f)=f$ for every $f\in\fixpsi$.

Let 
$$L=\{g\in\ck: g=\pi(|h|^2)-|h|^2 \text{ for some }h\in\fixpsi\}.$$

If $g\in L$, then it follows from  (3) that 
$g\ge0$.
Moreover  $\Psi(g)\le g$ hence $\fil(g_U)\le g_U$.

Suppose now that there exists a point $z_0\in U$ such that 
$g(z_0)=0$ for every $g\in L$.
It then follows from (B)  that $g=0$ on $U$ hence on $K$ for every $g\in L$. 
This easily implies that $\fixpsi$ is closed under multiplication hence it equals $A$.
Indeed if $h_1$ and $\h2$ are functions in $\fixpsi$ then 
$\pi(h_1\h2)-h_1\h2$ can be written as a linear combination of elements from $L$ 
(see step 3).
It then follows that 
$\pi(h_1\h2)=h_1\h2$
therefore $\fixpsi=A$.

This shows that the map 
$\theta:\ck\to\ck$ defined above is multiplicative on $\ck$ and its range equals $A$.
Let $M$ be the maximal ideal space of $A$ 
and for each $h\in A$ let $\hat{h}\in\cm$ be its Gelfand transform.
Since $A$ is self-adjoint, 
the Gelfand transform is an isometric isomorphism from $A$ onto $\cm$
(the Banach algebra of all continuous, complex-valued functions on $M$).
It then follows that there exists a continuous one-to-one  map
$\gm: M\to K$
such that
$\widehat{\theta(f)}=f\circ\gm$ for every $f\in\ck$.
Moreover, since $A\subset\ck$ there exists a continuous surjective map 
$\lm:K\to M$ such that
$h=\hat{h}\circ\lm$ for every $h\in A$.

We now claim that $\gm(M)\subset\bdu$.
Suppose there exists $\alpha\in M$ such that $z=\gm(\alpha)\in U$.
Then 
there exists   a non-constant real valued function $h\in A$  
such that 
 $$\hat{h}(\alpha)=\sup \{h(x) : x\in K\}.$$
  Every such $h$ attains its maximum only on $\bdu$.
 However
$$ \hat{h}(\alpha)=h(\gamma(\alpha))=h(z)$$
This shows that $\gm(M)\subset\bdu$.

Let $z\in U$ and let $x=\gm\circ\lm(z)$. Then $x\in\bdu$.  
Let $h\in\fixpsi$ which attains its maximum on $K$ 
only at this point.
Then
$$h(x)=(h\circ\gm\circ\lm)(z)=(\hat{h}\circ\lm)(z)=h(z).$$
We get a contradiction.
This means that there is no $z_0\in U$ such that $g(z_0)=0$ for all $g\in L$.

 (6) It follows from (5) that 
$$\bdu=\{x\in K : g(x)=0 \text{ for every } g\in L\}.$$
It then follows that the closed ideal of $\ck$ generated by $L$ is precisely 
the ideal of all $f\in\ck$ which vanish identically on all points of $\bdu$.
Recall now that we already proved in part 3 that for any $g$ in the closed ideal generated by $L$ the sequence $\{\Psi^n(g)\}$ converges uniformly to $0$.
In particular, this holds true for  all functions $g$ of the form
$g=f-\theta(f)$ with $f\in\ck$.
This completes the proof of this theorem.

\end{proof}

This proof works as well for Markov semigroups with continuous parameter.
Examples of  Markov operators on complex domains which satisfy all the conditions in Theorem \ref{main} are given in \cite{AE}. 
As a matter of fact, the results and the methods used in \cite{AE} strongly motivated and inspired our research.
We have also used some ideas appearing in \cite{Pr}.
Theorem \ref{main} 
can be used to give an alternate proof for the main result in \cite{Liu}.

\end{document}